\documentclass[11pt]{article}
\usepackage{amsmath,amsfonts,amssymb, amsthm,enumerate,graphicx}
\usepackage{wrapfig}
\usepackage{lipsum}
\usepackage{tikz,authblk}
\usepackage{float}
\usepackage{hyperref}
\usepackage{subfig}
\usepackage{url}

\usetikzlibrary{shapes.geometric}
\newtheorem{theorem} {{\textsf{Theorem}}}
\newtheorem{proposition}[theorem]{{\textsf{Proposition}}}
\newtheorem{corollary}[theorem]{{\textsf{Corollary}}}

\newtheorem{definition}[theorem]{{\textsf{Definition}}}
\newtheorem{remark}[theorem]{{\textsf{Remark}}}

\newtheorem{lemma}[theorem]{{\textsf{Lemma}}}
\newtheorem{speculation}[theorem]{{\textsf{Speculation}}}

\newcommand{\Sp}{\mathbb{S}}

\begin{document}
\title{Minimal Simplicial Degree $d$ Maps from Genus $g$ Surfaces to the Torus}
\author{Biplab Basak$^1$ and Ayushi Trivedi}
\date{}
\maketitle
\vspace{-10mm}
\begin{center}

\noindent {\small Department of Mathematics, IIT Delhi, Hauz Khas, New Delhi 110016.}

\noindent {\small {\em E-mail addresses:} \url{biplab@iitd.ac.in}, 
\url{Ayushi.Trivedi@maths.iitd.ac.in}}

\medskip

\date{\today}
\end{center}
\footnotetext[1]{Corresponding author}
\hrule

\begin{abstract}
The degree of a map between orientable manifolds is a fundamental concept in topology, offering deep insights into the structure of the manifolds and the nature of the corresponding maps. This concept has been extensively studied, particularly in the context of simplicial maps between orientable triangulable spaces. In 1982, Gromov proved that if degree $d$ maps exist from a genus $g$ orientable surface to a genus $h$ orientable surface for every $d \in \mathbb{Z}$, then $h$ must be 0 or 1.

Recently, degree $d$ self-maps on spheres, particularly on genus 0 surfaces, have been investigated. In this paper, we focus on the unique minimal 7-vertex triangulation of the torus. We construct simplicial degree $d$ maps from a triangulation of a genus $g$ surface to the 7-vertex triangulation of the torus for $g \geq 1$. Our construction of degree $d$ maps is minimal for every $d$ when $g = 1,2$. If $g \geq 3$, then our construction remains minimal for $|d| \geq 2g - 1$.

We believe that this concept will be highly useful in combinatorial topology, as it leads to several intriguing open research problems. In the final section, we propose some of these open research problems. 
\end{abstract}

\noindent {\small {\em MSC 2020\,:}  Primary 05E45; Secondary 05B45, 05C10, 55M25, 52B70, 57Q15.
	
\noindent {\em Keywords:} Simplicial map, Degree of simplicial maps, Triangulations of surfaces}

\section{Introduction}
The relationship between topology and combinatorics has been a fruitful area of study, especially in the context of manifold triangulations. These triangulations offer a combinatorial approach to analyzing topological spaces by breaking them down into simplices. This decomposition not only aids in theoretical exploration but also enhances computational methods, establishing triangulations as a crucial tool in combinatorial topology.

A key concept in topology is the degree of a map \( f: X \to Y \) between closed orientable \( n \)-manifolds, serving as an essential invariant in topology, geometry, and mathematical physics. Originally introduced by Gauss in the context of algebra, the degree was first defined for smooth maps and later extended to continuous maps via homotopy theory (see Milnor \cite{Milnor}). The foundational works of Hopf \cite{Hopf1928, Hopf1930}, Olum \cite{Olum1953}, and Epstein \cite{Epstein1966} further developed the theory. For dimension two, important results appear in \cite{Kneser, Skora}, while \cite{BrownSchirmer} explores the connection between Nielsen root theory and degree theory, providing new insights.

A key result in this area, established in \cite{Ryabichev2024}, states that for given closed (possibly nonorientable) surfaces \( M \) and \( N \), if a map \( f : M \to N \) has geometric degree \( d > 0 \), then \( \chi(M) \leq d \cdot \chi(N) \). In homology theory, the degree of a map \( f: X \to Y \) is directly computable since \( H_n(X, \mathbb{Z}) = H_n(Y, \mathbb{Z}) = \mathbb{Z} \). The induced homomorphism \( f_*: H_n(X, \mathbb{Z}) \to H_n(Y, \mathbb{Z}) \) satisfies \( f_*([a]) = d \cdot [b] \), where \( [a] \) and \( [b] \) are generators of the \( n \)th homology groups of \( X \) and \( Y \), respectively, and \( d \) is the degree of \( f \). Triangulations play a crucial role in computing homology and degree maps. In \cite{Fan1967}, a combinatorial theorem on simplicial maps from an orientable \( n \)-pseudomanifold to an \( m \)-sphere with the octahedral triangulation was established. Results concerning maps between \( n \)-spheres can be found in \cite{Whitney}.  In \cite{Madahar2001}, a minimal triangulation \( K \) was constructed such that a simplicial degree \( d \) self-map \( f: K \to \mathbb{S}^2_4 \) exists. More recently, in \cite{BGT25}, similar constructions were extended to higher dimensions, providing minimal triangulations that admit simplicial degree \( d \) self-maps \( f: K \to \mathbb{S}^n_{n+2} \).

Building upon these results, Gromov \cite{Gromov1982} proved that if degree \( d \) maps exist from a genus \( g \) orientable surface to a genus \( h \) orientable surface for every \( d \in \mathbb{Z} \), then \( h \) must be either 0 or 1. Recent studies have further examined degree \( d \) self-maps on spheres, particularly in the case of genus zero surfaces. In this paper, we focus on the unique minimal 7-vertex triangulation of the torus and construct simplicial degree \( d \) maps from triangulations of genus \( g \) surfaces to this 7-vertex torus triangulation for \( g \geq 1 \). Our construction yields minimal degree \( d \) maps for all \( d \) when \( g = 1,2 \) and remains minimal for \( |d| \geq 2g - 1 \) when \( g \geq 3 \).

This study contributes to combinatorial topology by offering new insights and generating several intriguing open problems. In the final section, we outline these problems, hoping they will inspire further research in this domain.

\section{Preliminaries}
\subsection{Basic Terminologies}
    A {\em simplicial complex} $K$ is a finite collection of simplices in $\mathbb{R}^m$ for some $m \in \mathbb{N}$ such that for any simplex $\sigma$, all of its faces are in $K$, and for any two simplices $\sigma$ and $\tau$ in $K$, either $\sigma \cap \tau$ is empty or a face of both simplices. We assume that the empty set is the face of every simplicial complex. The geometric carrier of $K$ is a compact polyhedron $\|K\|$ and is defined as $\|K\|:= \bigcup_{\sigma \in K} \sigma$.

 A triangulation of a polyhedron $X$ is a simplicial complex $K$ with a PL-homeomorphism between $X$ and $\|K\|$. If $M$ is a topological $n$-manifold and $\|K\|$ is homeomorphic to $M$ for an $n$-simplicial complex $K$, then we say that $K$ is a triangulation of $M$ or a triangulated $n$-manifold. Let $K$ be a triangulated $n$-manifold with a connected boundary. Then $\partial(K)$ is the boundary of $K$ and is a triangulated $(n-1)$-manifold. Furthermore, $\partial(K)$ is the collection of $(n-1)$-faces that are contained in exactly one $n$-face.
 If $K$ and $L$ are triangulable $n$ manifolds, then their connected sum, denoted by $K\#L$, is obtained by removing an interior of $n$-simplex from each of $K$ and  $L$, and then gluing the resulting boundaries of these simplices together, the resulting space $K\#L$ is a triangulable $n$-manifold.
 A simplicial complex $X$ is called vertex minimal if $f_0(X) \leq f_0(Y)$ for all triangulations $Y$ of the topological space $\|X\|$. An $n$-dimensional simplicial complex $X$ is called facet minimal if $f_n(X) \leq f_n(Y)$ for all triangulations $Y$ of the topological space $\|X\|$. By $Card(S)$, we mean the cardinality of a set $S$.

    \subsection{Degree of a simplicial map}

Let $M$ and $N$ be closed, orientable, triangulated $n$-manifolds, and let $f: M \rightarrow N$ be a simplicial map. Since the $n$-th homology groups $H_n(M, \mathbb{Z}) $ and $H_n(N, \mathbb{Z}) $ are both isomorphic to $ \mathbb{Z}$, and the simplicial map $f$ induces a homomorphism $f_*: H_n(M, \mathbb{Z}) \to H_n(N, \mathbb{Z}) $. Let $ f_*([a]) = d \cdot [b] $, where $ H_n(M, \mathbb{Z}) = \langle [a] \rangle $ and $ H_n(N, \mathbb{Z}) = \langle [b] \rangle $, and $ d \in \mathbb{Z} $. The integer $d$ is called the degree of the map $f$.

A simplicial map between two simplicial complexes $K$ and $L$ is a map $f: K \rightarrow L$ such that if $[v_1 v_2 \cdots v_{m+1}]$ is a simplex in $K$, then $[f(v_1) f(v_2) \cdots f(v_{m+1})]$ is a simplex in $L$. Suppose an $n$-simplex $\sigma = [v_1 v_2 \cdots v_{n+1}]$ is positively oriented with respect to the ordering $v_1 < v_2 < \cdots < v_{n+1}$, then the simplices formed using even permutations of the arrangement $(v_1, v_2, \ldots, v_{n+1})$ are positively oriented, and the simplices formed using odd permutations of the arrangement $(v_1, v_2, \ldots, v_{n+1})$ are negatively oriented.
\smallskip

 An $n$-manifold $K$ is coherently oriented if all of its $n$-simplices are oriented in the following way:
If $\sigma_1 = [v_1 v_2 \cdots v_{i-1} v_i v_{i+1} \cdots v_{n+1}]$ and $\sigma_2 = [v_1 v_2 \cdots v_{i-1} v'_i v_{i+1} \cdots v_{n+1}]$ are two $n$-simplices sharing an $(n-1)$-simplex $[v_1 v_2 \cdots v_{i-1} v_{i+1} \cdots v_{n+1}]$, then $\sigma_1$ and $\sigma_2$ are oppositely oriented with respect to the orderings $v_1 < v_2 < \cdots < v_i < \cdots < v_{n+1}$ and $v_1 < v_2 < \cdots < v'_i < \cdots < v_{n+1}$, respectively. An $n$-manifold is said to be orientable if it admits coherent orientation. The symbol $\Sigma_g$ represents the orientable surface with genus $g$.   
\smallskip

\begin{remark}\label{orientation}
    {\rm Consider the set of vertices $\{v_1, v_2,\ldots,v_{7}\}$ representing a minimal vertex triangulation of the torus, denoted as $\mathbb{T}{^2}$. The triangulation consists of the following 14 triangles: $[v_1v_2v_4]$, $[v_2v_4v_5]$, $[v_2v_3v_5]$, $[v_3v_5v_6]$, $[v_1v_5v_6]$, $[v_1v_2v_6]$, $[v_2v_6v_7]$, $[v_2v_3v_7]$, $[v_1v_3v_7]$, $[v_1v_5v_7]$, $[v_4v_5v_7]$, $[v_4v_6v_7]$, $[v_3v_4v_6]$, $[v_1v_3v_4]$. This configuration provides a minimal triangulation of the torus $\mathbb{T}^{2}$, which can be arranged within a triangulated $4$-gon, see Figure \ref{fig 1}.

    Consider $\sigma_1=[v_1v_2v_4]$ as a $2$-simplex in $\Sigma_1$, positively oriented with respect to the order $v_1<v_2<v_4$. The simplex $\sigma_2 =[v_{1}v_3v_4]$ will then be negatively oriented with respect to the order $v_1<v_3<v_4$, since both simplices share the edge $[v_1v_4]$ with the same ordering. Similarly, for any $2$-simplex $\tau$ in a minimal vertex triangulation of $\Sigma_1$, its orientation can be determined relative to the order $v_1 < v_2 < \cdots < v_{7}$. We classify positively oriented simplices in a triangulated orientable surface as positive simplices and negatively oriented ones as negative simplices. Extending this idea to a  triangulated $\Sigma_g$, if we assign an orientation to a single 2-simplex as either positive or negative, we can extend this orientation to all other simplices of $\Sigma_g$ accordingly. Let $K$ be a triangulated orientable surface with genus $g$, and let $f: K \rightarrow \mathbb{T}{^2}$ be a simplicial map. If $f$ is not surjective, then the degree of the map $f$ is zero. Furthermore, if $f$ maps a $2$-simplex to a $1$-simplex or a $0$-simplex, then $f(\sigma)$ does not contribute to the second homology group of $\mathbb{T}{^2}$.}

    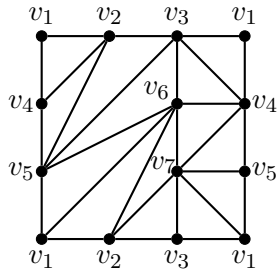
\begin{figure}[H]
\tikzstyle{ver}=[]
\tikzstyle{vertex}=[circle, draw, fill=black!100, inner sep=0pt, minimum width=4pt]
\tikzstyle{edge} = [draw,thick,-]
\centering

\begin{tikzpicture}[scale=0.9]
\begin{scope}[shift={(0,0)}]
\foreach \x/\y/\z in {0/0/v_1,1/0/v_2,2/0/v_3,3/0/v_12,0/1/v_5,0/2/v_4,0/3/v_11,3/1/v_14,3/2/v_13,3/3/v_8,2/3/v_9,1/3/v_10,2/2/v_6,2/1/v_7}{
\node[vertex] (\z) at (\x,\y){};}

\foreach \x/\y/\z in {0/-0.3/v_1,1/-0.3/v_2,2/-0.3/v_3,3/-0.3/v_1, 0/3.3/v_1,1/3.3/v_2,2/3.3/v_3,3/3.3/v_1, -0.3/1/v_5,-0.3/2/v_4,1.8/1.1/v_7,1.7/2.2/v_6,3.3/1/v_5,3.3/2/v_4}{\node[ver] () at (\x,\y){$\z$};}

\foreach \x/\y in {v_1/v_2,v_2/v_3,v_3/v_12,v_12/v_14,v_14/v_13,v_13/v_8,v_8/v_9,v_9/v_10,v_10/v_11,v_11/v_4,v_4/v_5,v_5/v_1,v_3/v_7,v_7/v_6,v_6/v_9,v_10/v_4,v_10/v_5,v_9/v_5,v_9/v_13,v_6/v_5,v_6/v_1,v_6/v_2,v_6/v_13,v_7/v_13,v_7/v_14,v_7/v_12,v_7/v_2}{
\path[edge] (\x) -- (\y);}

\end{scope}

\end{tikzpicture}
\caption{Triangulation of Torus}
 \label{fig 1}
 \end{figure}
\end{remark}
Let $\alpha^{+}(\sigma)$ and $\alpha^{-}(\sigma)$ denote the sets of positive and negative simplices in $K$ mapping to the simplex $\sigma$ in $\mathbb{T}^{2}$, respectively. We define the algebraic number of a simplex $\sigma$ in $\mathbb{T}^{2}$ as follows:

\[
\text{alg}(\sigma) = 
\begin{cases}
Card(\alpha^{+}(\sigma)) - Card(\alpha^{-}(\sigma)), & \text{if } \sigma \text{ is positive} \\
Card(\alpha^{-}(\sigma)) - Card(\alpha^{+}(\sigma)), & \text{if } \sigma \text{ is negative}
\end{cases}
\]

Let the simplicial map $f: K \rightarrow $ $\mathbb{T}^{2}$ induce the homomorphism $f_*: H_2(K, \mathbb{Z}) \to H_2(\mathbb{T}^{2}, \mathbb{Z})$ such that $f_*([a]) = d \cdot [b]$, where $H_2(K, \mathbb{Z}) = \langle [a] \rangle$ and $H_2(\mathbb{T}^{2}, \mathbb{Z}) = \langle [b] \rangle$, and $d \in \mathbb{Z}$. Then $a$ represents the sum of all positive $2$-simplices in $K$ minus the sum of all negative $2$-simplices in $K$, and $b$ represents the sum of all positive $2$-simplices in $\mathbb{T}^{2}$ minus the sum of all negative $2$-simplices in $\mathbb{T}^{2}$. Furthermore, $d$ equals $\text{alg}(\sigma)$ for every $\sigma \in \mathbb{T}^{2}$.

\begin{proposition}\label{Composition}
  Let $K$, $L$, and $M$ be triangulated closed orientable $n$-manifolds. If $f: K \rightarrow L$ and $g: L \rightarrow M$ are simplicial maps with degrees $d$ and $d'$, respectively, then the degree of the composite map $g \circ f$ is $d'\cdot d$.
\end{proposition}

\subsection{Simplicial Volume}
\begin{definition}[Simplicial Volume]  {\rm
Let \( M \) be an oriented, closed, connected \( n \)-dimensional manifold. The {\em simplicial volume} of \( M \) (also known as the Gromov norm of \( M \)) is defined as  

\[
\|M\|^V := \inf \{ |c|_1 \mid c \in C_n(M; \mathbb{R}) \text{ is a fundamental cycle of } M \} \in \mathbb{R}_{\geq 0},
\]

where \( [M] \in H_n(M; \mathbb{R}) \) denotes the fundamental class of \( M \) with real coefficients.  }
\end{definition}

The \( \ell_1 \)-norm on the singular chain complex \( C_*(X; \mathbb{R}) \) with real coefficients is defined as follows: for a chain  

  \[
  c = \sum_{j=0}^{k} a_j \cdot \sigma_j \in C_*(X; \mathbb{R})
  \]

  (expressed in reduced form), its \( \ell_1 \)-norm is given by  

  \[
  |c|_1 := \sum_{j=0}^{k} |a_j|.
  \]

The \( \ell_1 \)-semi-norm on singular homology \( H_*(X; \mathbb{R}) \) is induced by the \( \ell_1 \)-norm on chains. More explicitly, for a topological space \( X \) and a homology class \( \alpha \in H_*(X; \mathbb{R}) \), the semi-norm is defined as  

  \[
  \|\alpha\|_1^V := \inf \{ |c|_1 \mid c \in C_*(X; \mathbb{R}) \text{ is a cycle representing } \alpha \}.
  \]

\begin{proposition}[\cite{Gromov1982}]\label{proposition}
Let \( f: M \to N \) be a map between oriented, closed, connected manifolds of the same dimension. Then, the simplicial volume of \( M \) and \( N \) satisfies the inequality  

\[
| \deg f | \cdot \| N \|^V \leq \| M \|^V
\]

where \( \| M \|^V \), \( \| N \|^V \) denote the simplicial volume of \( M \) and \( N \) , respectively. 
\end{proposition}
It is well known (cf. \cite{Gromov1982}) that for any oriented, closed, connected surface \( S_g \) of genus \( g \geq 1 \), the simplicial volume is given by  

\[
\| S_g \|^V = 4g - 4.
\]
\begin{remark}\label{rem 5}
{\rm
Let $\Sigma_{g_1}$ and $\Sigma_{g_2}$ be triangulated orientable surfaces of genus $g_1$ and $g_2$, respectively, where $g_1$, $g_2$ $\in \mathbb{N}$ and $g_1$, $(g_2>1)$ . If we replace $M$ and $N$ in Proposition \ref{proposition} with $\Sigma_{g_1}$ and $\Sigma_{g_2}$, respectively, then the degree of any simplicial map $f:\Sigma_{g_1} \rightarrow \Sigma_{g_2}$, satisfies \[
| \deg f |  \leq  \frac{g_1-1}{g_2-1}.
\]
The following observations can be made from this inequality:
\begin{enumerate}[$(i)$]
   
    \item If $g_1=g_2=g$ with $g>1$, then the only possible degrees for simplicial maps from  $\Sigma_g$ to $\Sigma_g$ are $\pm 1$ and 0.
    \item If $g_1> g_2>1$, then there exists a degree $d$ map from $\Sigma_{g_1}$ to $\Sigma_{g_2}$, where $d\in [-\lfloor \frac{g_1-1}{g_2-1}\rfloor, \lfloor \frac{g_1-1}{g_2-1}\rfloor]\cap \mathbb{Z}$.
    \item If $g_2>g_1$, the only possible simplicial map from $\Sigma_{g_1}$ to $\Sigma_{g_2}$ has degree 0.

\end{enumerate}
}
\end{remark}

\section{Degree $d$ maps from  genus $g$ surfaces to the torus}
For each degree $d\in \mathbb{Z}$, there exists a degree $d$ simplicial map from $\Sigma_g$ to $\Sigma_1$, with $g \geq 1$.
 \begin{lemma}
  Let $\mathbb{T}^{2}$ be a $7$-vertex triangulated torus. If the orientation is fixed such that the simplex $[v_1v_2v_4]$ is positively oriented, there does not exist a simplicial map of degree $-1$ from $\mathbb{T}^{2}$ to itself.  
 \end{lemma}
 \begin{proof}
Consider all bijective simplicial maps from the $7$-vertex triangulated torus to itself. The 42 generators of the automorphism group, which induce simplicial maps from the torus to itself, are given by:

\begin{center}
Identity, (2,3,5)(4,7,6), (2,4,3,7,5,6), (2,5,3)(4,6,7), (2,6,5,7,3,4), (2,7)(3,6)(4,5), 
(1,2)(3,7)(4,6), (1,2,3,4,5,6,7), (1,2,4)(3,6,5), (1,2,5,7,6,3), (1,2,6)(4,7,5), (1,2,7,4,3,5), (1,3,6,7,5,2), (1,3)(4,7)(5,6), (1,3,5,7,2,4,6), (1,3,7)(2,5,4), (1,3,2,6,4,5), (1,3,4)(2,7,6), (1,4,2)(3,5,6), (1,4,5,3,7,6), (1,4)(2,3)(5,7), (1,4,7,3,6,2,5), (1,4,3)(2,6,7), (1,4,6,5,2,7), (1,5,3,4,7,2), (1,5,7)(3,6,4), (1,5,4,6,2,3), (1,5)(2,4)(6,7), (1,5,2,6,3,7,4), (1,5,6)(2,7,3), (1,6,2)(4,5,7), (1,6,7,3,5,4), (1,6,5)(2,3,7), (1,6,3,2,4,7), (1,6)(2,5)(3,4), (1,6,4,2,7,5,3), (1,7,6,5,4,3,2), (1,7,5)(3,4,6), (1,7,4,2,3,6), (1,7,3)(2,4,5), (1,7,2,5,6,4), (1,7)(2,6)(3,5). 
\end{center}
By relabeling the vertices $v_i$ $(1\leq i\leq 7)$ of the triangulation in Figure \ref{fig 1} with $i$, we observe that all these maps preserve orientation. That is, each bijective simplicial map sends a positively oriented simplex to another positively oriented simplex and a negatively oriented simplex to another negatively oriented simplex. Since a simplicial map of degree $-1$ must reverse orientation, none of these $42$ bijective simplicial maps can have degree $-1$. Therefore, no simplicial self-map of the $7$-vertex triangulated torus attains degree $-1$ while preserving the given orientation.
 \end{proof}

 \begin{remark}\label{negative degree}
 {\rm 
To construct a simplicial map of degree $-1$ from $7$-vertex triangulated torus $\mathbb{T}^{2}$ to itself, it is necessary to reverse the orientation of the triangulated torus in the domain. Under this transformation, the simplex $[v_1v_2v_4]$ in the domain becomes a negatively oriented simplex, which implies that all positive simplices become negative and vice versa. Thus, the identity map $\mathbb{I}$ now induces a simplicial map of degree \(-1\). More generally, let $f:\Sigma_g \rightarrow \mathbb{T}^{2}$ be a simplicial map, where $\Sigma_g$ is a triangulated orientable surface of genus $g$ ($g \geq 1$), and $\mathbb{T}^{2}$ is a $7$-vertex triangulation of the torus. If $f$ has a degree $d$ simplicial map, then $f\circ \mathbb{I}$ defines a simplicial map of degree $-d$ from the triangulated $\Sigma_g$ to $\mathbb{T}^{2}$.
}
\end{remark}
    
\begin{theorem}\label{thm1}

    For every $|d|\geq 2g-1$, where $g\geq 1$, there exists a minimal degree $d$ simplicial map from a $(7|d|+2-2g)$-vertex triangulated genus $g$ orientable surface to a $7$-vertex triangulated genus $1$ orientable surface.
\end{theorem}

\begin{proof}

As we know, the genus $1$ orientable surface is a torus, and its minimal vertex triangulation is described in Remark \ref{orientation}. We define a simplicial map of negative degree, as in Remark \ref{negative degree}, using a simplicial map of positive degree composed with a simplicial map of degree $-1$. Any non-surjective simplicial map is considered to have degree $0$. Therefore, in this proof, our primary focus is on simplicial maps of positive degree $d$. 

To define a simplicial map for $d>0$, we first construct a triangulation of $\Sigma_g$. We begin with a regular $4g$-gon, where opposite sides are identified and all corner vertices coincide, denoted as $u_{(1,1)}$. Within this $4g$-gon, we introduce new vertical internal $2(g-1)$ edges, subdividing it into $2g-1$ quadrilaterals. We label the sides of the $4g$-gon as 
$\alpha_1\beta_1\alpha_2\beta_2\cdots\alpha_g\beta_g\alpha_1^{-1}\beta_1^{-1}\cdots\alpha_g^{-1}\beta_g^{-1}$, where each $\alpha_i$ and $\beta_i$ represent the loops that generate the fundamental group of the surface. Moreover, $\alpha_i^{-1}$ and $\beta_i^{-1}$  indicate the same loops traversed in the reverse direction, corresponding to the inverse elements in the fundamental group. Choose a starting side, which is labeled $\alpha_1$, in such a way that two sides of the middle rectangle of the $4g$-gon become $\beta_{g/2}$ and $\beta_{g/2}^{-1}$ when $g$ is even, and $\alpha_{(g+1)/2}$ and $\alpha_{(g+1)/2}^{-1}$ when $g$ is odd.

For $d\geq 2g-1$, suppose $d = (2g-1) + l$, where $l \in \mathbb{N}\cup\{0\}$. We further refine the $g^{\text{th}}$ quadrilateral (the middle rectangle) by introducing $l$ additional vertical lines. 
Since opposite sides of the $4g$-gon are identified, the endpoints of these vertical lines correspond to the same set of vertices, labeled as $u_{(1,j)}$ for $2\leq j\leq l+1$. 
This process partitions the quadrilateral into $l+1$ new rectangles, resulting in a total of $2g-1+l$ quadrilaterals in the $4g$-gon. To obtain a triangulation of $\Sigma_g$, we triangulate each of these quadrilaterals in the same way as we do in the torus, as shown in Figure \ref{fig 1}.

To assign vertices to the triangulated surface $\Sigma_g$, we categorize them based on their placement on the left side, upper side, and interior part of each quadrilateral in the construction. For the first quadrilateral, the vertices on the left side are $u_{(1,1)}$, $u_{(5,1)}$, $u_{(4,1)}$, $u_{(1,1)}$, while those on the upper side are  $u_{(1,1)}$, $u_{(2,1)}$, $u_{(3,1)}$, $u_{(1,1)}$, and the interior part consists of the vertices $u_{(6,1)}$, $u_{(7,1)}$. Similarly, for the $i^{\text{th}}$ quadrilateral, where $i\leq g-1$, the left side has vertices $u_{(1,1)}$, $u_{(5,i)}$, $u_{(4,i)}$, $u_{(1,1)}$, the upper side has $u_{(1,1)}$, $u_{(2,i)}$, $u_{(3,i)}$, $u_{(1,1)}$, and the interior part contains $u_{(6,i)}$, $u_{(7,i)}$. For the $g^{\text{th}}$ quadrilateral, the left side follows the same pattern as $i\leq g-1$, with vertices $u_{(1,1)}$, $u_{(5,g)}$, $u_{(4,g)}$, $u_{(1,1)}$, while the upper side is assigned $u_{(1,1)}$, $u_{(2,g)}$, $u_{(3,g)}$, and $u_{(1,2)}$, and the interior part remains  $u_{(6,g)}$, $u_{(7,g)}$. Moving to the ${(g+j)}^{\text{th}}$ quadrilateral, where $1\leq  j\leq l-1$, the left side contains $u_{(1,j+1)}$, $u_{(5,g+j)}$, $u_{(4,g+j)}$, $u_{(1,j+1)}$, the upper side consists of  $u_{(1,j+1)}$, $u_{(2,g+j)}$, $u_{(3,g+j)}$, and $u_{(1,j+2)}$, and the interior part contains $u_{(6,g+j)}$, $u_{(7,g+j)}$. For the ${(g+l)}^{\text{th}}$ quadrilateral, the left side is given by $u_{(1,l+1)}$, $u_{(5,g+l)}$, $u_{(4,g+l)}$, $u_{(1,l+1)}$, while the upper side consists of $u_{(1,l+1)}$, $u_{(2,g+l)}$, $u_{(3,g+l)}$, $u_{(1,1)}$, and the interior part has $u_{(6,g+l)}$, $u_{(7,g+l)}$. For the $(g+j)^{\text{th}}$ quadrilateral, where $l+1\leq  j\leq g-1+l$, the left side is assigned $u_{(1,1)}$, $u_{(5,g+j)}$, $u_{(4,g+j)}$, $u_{(1,1)}$, while the upper side consists of  $u_{(1,1)}$, $u_{(2,g+j)}$, $u_{(3,g+j)}$, $u_{(1,1)}$, and the interior part has $u_{(6,g+j)}$, $u_{(7,g+j)}$. In terms of the relationships between quadrilaterals, the left side of each $i^{\text{th}}$ quadrilateral serves as the right side of $(i-1)^{\text{th}}$ quadrilateral, except for $i=1$, where the left side of the first quadrilateral corresponds to the right side of the $(2g-1+l)^{th}$ quadrilateral. The bottom side of each $i^{\text{th}}$ quadrilateral aligns with the upper side of the $(2g-i+l)^{\text{th}}$ quadrilateral, and vice versa, for $1\leq i\leq g-1$. For the remaining quadrilateral, the upper and bottom sides remain the same. Thus, we have systematically assigned vertices to the triangulated surface $\Sigma_g$. 
Each quadrilateral contains six distinct vertices, except $u_{(1,1)}$ and $u_{(1,j+1)}$, where $1\leq j\leq l$. Consequently, the total number of vertices in the triangulated $\Sigma_g$ is given by $6(2g-1+l)+(l+1)$ $=$ $7d+2-2g$. An example of such a triangulation of $\Sigma_2$ is shown in Figures \ref{fig 2} and \ref{fig 3}, where the refined middle rectangle is clearly visible, and the surface is subdivided into five quadrilaterals by the construction described above. Using this triangulation, we define a simplicial map of degree $2$ from $\Sigma_2$ to $\mathbb{T}^2$, as explained in the following paragraphs.

\begin{figure}[H]
\tikzstyle{ver}=[]
\tikzstyle{vertex}=[circle, draw, fill=black!60, inner sep=0pt, minimum width=3pt]
\tikzstyle{edge} = [draw,black,thick,-]
\centering

\begin{tikzpicture}[scale=0.5]
\begin{scope}[shift={(0,0)}]
\foreach \x/\y/\z in {0/0/l_1,1/-1/l_2,2/-2/l_3,3/-3/l_4,7.242/-3/l_7,8.242/-2/l_8,9.242/-1/l_9,10.242/0/l_10,10.242/1.414/l_11,10.242/2.828/l_12,10.242/4.242/l_13,9.242/5.242/l_14,8.242/6.242/l_15,7.242/7.242/l_16,3/7.242/l_19,2/6.242/l_20,1/5.242/l_21,0/4.242/l_22,0/2.828/l_23,0/1.4141/l_24,2/3.6/l_25,2.4/1.5/l_26,3/3.9/l_27,3/1.4/l_28,7.242/3.9/l_29,7.242/1.6/l_30,9.6/1.6/l_31,9.242/3.5/l_32}{
\node[vertex] (\z) at (\x,\y){};}

\foreach \x/\y/\z in {-0.75/0/u_{(1,1)},0.75/-1.34/u_{(2,5)},1.7/-2.3/u_{(3,5)},2.6/-3.3/u_{(1,1)}, 7.57/-3.4/u_{(1,1)},8.7/-2.45/u_{(2,1)},9.7/-1.45/u_{(3,1)},10.7/-0.45/u_{(1,1)},11.1/1.3/u_{(5,1)},11.1/2.9/u_{(4,1)},11.1/4.3/u_{(1,1)},10.1/5.2/u_{(3,5)},9.1/6.242/u_{(2,5)},8.1/7.3/u_{(1,1)},2.3/7.3/u_{(1,1)},1.3/6.3/u_{(3,1)},0.3/5.3/u_{(2,1)},-0.8/4.242/u_{(1,1)},-0.8/2.828/u_{(4,1)},-0.8/1.4141/u_{(5,1)}}{\node[ver] () at (\x,\y){{\fontsize{8pt}{7.6pt}\selectfont $\z$}};}
\foreach \x/\y/\z in {1.75/1.4/u_{(7,1)},1.8/3.9/u_{(6,1)},9.1/3.8/u_{(6,5)},9.1/1.8/u_{(7,5)}} {
    \node[ver] at (\x,\y) {\fontsize{4pt}{1pt}\selectfont $\z$};
}

 \foreach \x/\y in {l_1/l_2,l_2/l_3,l_3/l_4,l_4/l_7,l_7/l_8,l_8/l_9,l_9/l_10,l_10/l_11,l_11/l_12,l_12/l_13,l_13/l_14,l_14/l_15,l_15/l_16,l_16/l_19,l_19/l_20,l_20/l_21,l_21/l_22,l_22/l_23,l_23/l_24,l_24/l_1, l_3/l_26, l_26/l_25,l_25/l_20,
 l_4/l_19,l_16/l_7,l_31/l_9, l_31/l_32, l_32/l_14,l_21/l_23,l_21/l_24,l_24/l_20,l_20/l_27,l_25/l_24,l_25/l_1,l_25/l_2,l_26/l_2,l_26/l_4,l_26/l_28,l_26/l_27,l_25/l_27,l_15/l_29,l_15/l_30,l_30/l_14,l_14/l_12,l_32/l_30,l_32/l_7,l_32/l_8,l_31/l_8,l_32/l_12,l_12/l_31,l_31/l_11,l_31/l_10}{
 \path[edge] (\x) -- (\y);}
  
    \fill[gray!20] (3.1,-2.9) -- (3.1,7.1) -- (7.15,7.1) -- (7.15,-2.9) -- cycle;
    
\end{scope}

\end{tikzpicture}
\caption{Triangulation of $\Sigma_2$ for a degree $5$ simplicial map}
\label{fig 2}
\end{figure}
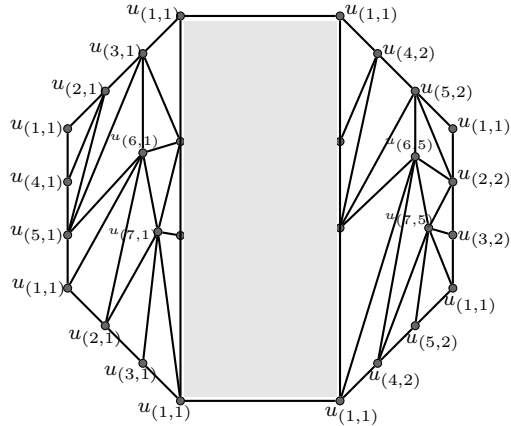
  \begin{figure}[H]
\tikzstyle{ver}=[]
\tikzstyle{vertex}=[circle, draw, fill=black!60, inner sep=0pt, minimum width=2pt]
\tikzstyle{edge} = [draw, black!80, thick, -]
\centering

\begin{tikzpicture}[scale=0.56]
\begin{scope}[shift={(0,0)}]
\foreach \x/\y/\z in {3/-3/l_12,4/-3/l_22,5.828/-3/l_32,7.242/-3/l_13,8/-3/l_23,10.07/-3/l_33,11.484/-3/s_2,12.898/-3/n_1,14.312/-3/n_2,15.726/-3/n_3,
15.726/7.242/n_11,14.312/7.242/n_12,12.898/7.242/n_13,11.484/7.242/s_1,10.07/7.242/u_33,8.656/7.242/u_23,7.242/7.242/u_13,5.828/7.242/u_32,4.414/7.242/u_22,3/7.242/r_1,3/1.414/u_52,3/4.242/u_42,5.828/1.414/u_72,5.828/4.242/u_62,7.242/1.6/u_53,7.242/4.242/u_43,10.07/1.414/u_73,10.07/4.242/u_63,11.484/1.414/s_5,11.484/4.242/s_4,
14.312/4.242/m_1,14.312/1.414/m_2,15.726/4.242/m_3,15.726/1.414/m_4}{
\node[vertex] (\z) at (\x,\y){};}

\foreach \x/\y/\z in { 2.81/-3.3/u_{(1,1)} ,4.414/-3.3/u_{(2,2)},5.828/-3.3/u_{(3,2)},7.242/-3.3/u_{(1,2)} ,8.656/-3.3/u_{(2,3)},10.07/-3.3/u_{(3,3)},11.484/-3.3/u_{(1,3)},12.898/-3.3/u_{(2,4)},14.312/-3.3/u_{(3,4)},15.726/-3.3/u_{(1,1)},
8.656/7.542/u_{(2,3)},8.656/7.542/u_{(2,3)},10.07/7.542/u_{(3,3)},11.484/7.542/u_{(1,3)},12.898/7.542/u_{(2,4)},14.312/7.542/u_{(3,4)},15.726/7.542/u_{(1,1)},
7.4/7.542/u_{(1,2)},5.828/7.542/u_{(3,2)},
4.414/7.542/u_{(2,2)},2.999/7.542/u_{(1,1)},2.3/1.4/u_{(5,2)},2.3/4.2/u_{(4,2)},5.2/4.2/u_{(6,2)},
7.8/1.4/u_{(5,3)},9.4/4.2/u_{(6,3)}}
{\node[ver, black]  () at (\x,\y){{\fontsize{7pt}{4pt}\selectfont $\z$}};}

\foreach \x/\y/\z in {5.3/1.3/u_{(7,2)},7.8/4.2/u_{(4,3)},9.5/1.3/u_{(7,3)},12.1/4.2/u_{(4,4)}, 12.1/1.2/u_{(5,4)},13.7/1.3/u_{(7,4)},13.7/4.2/u_{(6,4)},16.3/1.3/u_{(5,5)},16.3/4.2/u_{(4,5)}}{\node[ver] () at (\x,\y){{\fontsize{7pt}{4pt}\selectfont $\z$}};}

  \foreach \x/\y in {l_12/l_22,l_22/l_32,l_32/l_13,l_13/l_23,l_23/l_33,l_33/s_2,s_2/n_3,n_3/n_11,n_11/r_1,r_1/l_12,u_32/l_32,u_13/l_13,u_33/l_33,s_1/s_2,n_12/n_2,
u_22/u_42,u_52/u_22,u_52/u_32,u_32/u_62,u_62/u_72,u_72/l_32,u_62/u_52,u_62/l_12,u_62/l_22,l_13/u_53,u_53/u_72,u_72/u_43,u_43/u_53,u_43/u_62,u_43/u_32,u_43/u_13,l_13/u_72, u_72/l_22, u_23/u_43,u_53/u_23,u_53/u_33,u_33/u_63,u_63/u_73,u_63/s_4,u_73/l_33,u_63/u_53,u_63/l_13,u_63/l_23,s_2/s_5,s_5/s_4,s_4/s_1,u_73/l_23,u_33/s_4,s_4/u_73,u_73/s_5,u_73/s_2,n_13/s_4,n_13/s_5,n_12/s_5,m_1/s_5,m_1/s_2,m_1/n_1,m_2/n_1,m_2/n_3,m_2/m_4,m_2/m_3,m_1/m_3,n_12/m_3,n_11/m_3,m_3/m_4,m_4/n_3}{
 \path[edge] (\x) -- (\y);}
\end{scope}

\end{tikzpicture}
\caption{Triangulation of shaded region of $\Sigma_2$}
\label{fig 3}

\end{figure}
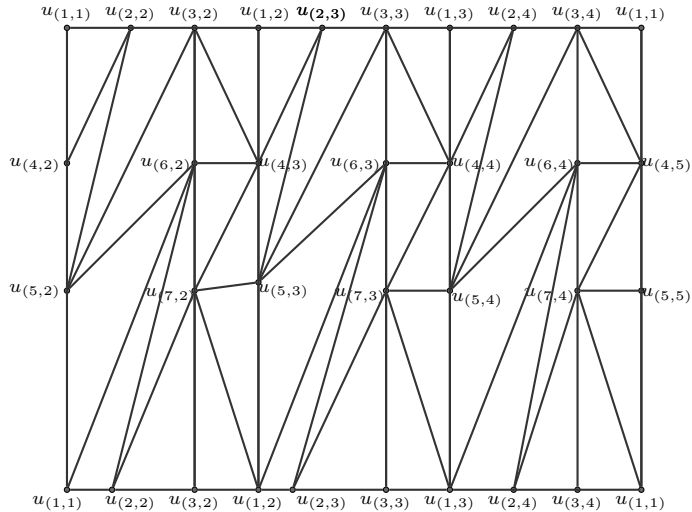

The minimal triangulation of $\Sigma_g$ corresponds to the degree $d$ simplicial map from $\Sigma_g$ to $\Sigma_1$. We can see from Euler's characteristic formula $n-e+f=2-2g$ and $2e=3f$, where $n$, $e$, and $f$ denote the total number of vertices, edges, and triangles, respectively. Combining these equations implies that $n=\frac{f}{2} + 2 -2g$. Since the minimal triangulation of $\Sigma_1$ consists of $14$ triangles and the preimage of each triangle under the degree $d$ map contains at least $d$ triangles, the total number of triangles in $\Sigma_g$ is at least $14d$, which ensures that at least $7d+2-2g$ vertices are required for the triangulation of $\Sigma_g$.

Define a simplicial map $f:\Sigma_g \rightarrow \Sigma_1$ by sending $u_{(i,j)}$ to $v_i$. This map induces a homomorphism $f_*:H_2(\Sigma_g, \mathbb{Z})\rightarrow H_2(\Sigma_1, \mathbb{Z})$. Consider $\sigma:=$ [$v_1v_2v_4$] as a positive simplex in $\Sigma_1$, and $\tau:= [u_{(1,1)}u_{(2,1)}u_{(4,1)}]$ as a positive simplex in $\Sigma_g$. From the construction, we observe that the preimage of each $2$-simplex in $\sigma \in \mathbb{T}^{2}$ contains exactly $d$ $2$-simplices from the triangulation of $\Sigma_g$, each carrying the same sign as its image simplex. The algebraic number of the $\sigma$ is $d$. Thus, this construction provides a degree $d$ map from $\Sigma_g$ to $\Sigma_1$.
\end{proof}

\begin{corollary}\label{Coro.}
There exists a minimal degree 
$d$ simplicial map from the triangulated torus to the $7$-vertex triangulated torus $\mathbb{T}^{2}$, as well as from a triangulated orientable surface of genus 
$2$ to the torus $\mathbb{T}^{2}$.
\end{corollary}
\begin{proof}
    The existence of a minimal degree $d$ simplicial map from the torus to itself, as well as a minimal degree $d$ simplicial map (for $d \geq 3$) from an orientable surface of genus $2$ to the torus, follows directly from Theorem \ref{thm1}.

For a minimal simplicial map of degree $d = 2$ from a triangulated $\Sigma_2$ to the torus $\mathbb{T}^{2}$, the number of vertices in $\Sigma_2$ must satisfy  $n \geq 12$, as explained in the proof of Theorem \ref{thm1}. To construct such a degree $2$ simplicial map, the preimage of each triangle in the torus must consist of at least two triangles from the triangulated $\Sigma_2$. If these two preimage triangles are to have the same orientation, they cannot share an edge; otherwise, their orientations would be reversed. Hence, they must either be disjoint or share only a vertex. To ensure minimal triangulation, we assume they share a vertex. 
For constructing the triangulation of $\Sigma_2$, consider the triangle $[v_1v_2v_4]$ in the torus. Its preimage under the simplicial map will consist of two triangles in $\Sigma_2$, say $[u_1u_2u_4]$ and $[u_1u'_2u'_4]$, which share only the vertex $u_1$. Mapping this single triangle requires five vertices. Next, consider the triangle $[v_1v_3v_4]$ in the torus, and note that the edges $u_1u_4$ and $u_1u'_4$ are already present in the constructing triangulation of $\Sigma_2$, both mapping to the edge $v_1v_4$. Now we need at least two vertices that map to $v_3$ and form two triangles. If only one vertex $u_3$ is added, the resulting two triangles would share the edge $u_1u_3$. Therefore, we need at least two vertices, say $u_3$ and $u'_3$, so that we can form two triangles, $[u_1u_3u_4]$, and $[u_1u'_3u'_4]$, both with the same orientation. At this stage, the constructing triangulation requires at least $7$ vertices. 
To construct preimages for additional triangles in the torus, such as $[v_1v_3v_7]$, $[v_1v_5v_7]$, and $[v_1v_5v_6]$, we again need at least $6$ more vertices. Since each edge in the triangulated surface $\Sigma_2$ must be shared by exactly two triangles, we can complete the triangulation by carefully reusing edges and introducing new vertices as needed. This observation shows that at least 13 vertices are required to form a triangulation of $\Sigma_2$ that admits a simplicial map of degree $2$ to the torus. A triangulation of $\Sigma_2$ with 13 vertices is illustrated in Figure \ref{fig 4}, and a detailed construction of the corresponding degree $2$ simplicial map to the torus $\mathbb{T}^2$ is given in Theorem \ref{thm 10}.

For a simplicial map of degree $d=1$ from a triangulated surface $\Sigma_2$ to a $7$-vertex triangulated torus $\mathbb{T}^{2}$, we consider a minimal triangulation of $\Sigma_2$, which exists with $10$ vertices. Let $K$ be such a triangulation of $\Sigma_2$, defined by the following set of simplices: $[[v_1v_2v_3],[v_1v_2v_4],[v_1v_3v_5]$, $[v_1v_4v_6],[v_1v_5v_7]$,
$[v_1v_6v_8],[v_1v_7v_8],[v_2v_3v_6],[v_2v_4v_8],[v_2v_5v_6]$,\\* $[v_2v_5v_9],[v_2v_7v_8],[v_2v_7v_{10}]$,
$[v_2v_9v_{10}],[v_3v_5v_{10}],[v_3v_6v_8],[v_3v_8v_9],[v_3v_9v_{10}],[v_4v_6v_{10}],[v_4v_7v_9]$,\\*$[v_4v_7v_{10}],
[v_4v_8v_9],[v_5v_6v_{10}],[v_5v_7v_9]$]. We define a simplicial map from $K$ to $\mathbb{T}^{2}$ by mapping the vertices as follows: $v_1 \rightarrow v_1$, $v_2\rightarrow v_2$, $v_3\rightarrow v_4$, $\{v_4, v_{10}\}\rightarrow v_6$, $v_5\rightarrow v_3$, $v_6\rightarrow v_5$, and $\{v_7,v_8, v_9\}\rightarrow v_7$. Now, consider $[v_1v_2v_3]$ as a positive simplex in $K$ and $[v_1v_2v_4]$ as a positive simplex in $\mathbb{T}^{2}$. 
 Under this simplicial map, the simplices $[v_2v_4v_8]$, $[v_2v_{10}v_{7}]$, and $[v_2v_{10}v_9]$ all map to $[v_2v_6v_7]$, in which two of them have the same sign as $[v_2v_6v_7]$. Therefore, the algebraic number of $[v_2v_6v_7]$ is 1. The preimage of every other triangle in $\mathbb{T}^{2}$ is a single triangle from $K$, and each of them has the same sign as its image triangle. Hence, for every triangle $\sigma$ in $\mathbb{T}^{2}$, $\text{alg}(\sigma)= 1$. Thus, the degree of this simplicial map from $K$ to $\mathbb{T}^{2}$ is 1.  
\end{proof}
\begin{theorem}\label{thm 10}
      For integers $|d|=g+i$, where $ 0 \leq i\leq g-2$, and $g\geq 2$, there exists a degree $d$ simplicial map from a $(6|d|+1)$-vertex triangulated genus $g$ orientable surface to a $7$-vertex triangulated genus $1$ orientable surface.
\end{theorem}
\begin{proof}
 To construct a simplicial map of degree $d\leq 0$ from $\Sigma_g$ to $\Sigma_1$, we follow the initial steps of the proof of Theorem $\ref{thm1}$. For a simplicial map of degree $d>0$, we construct a triangulation of $\Sigma_g$ and the triangulation of $\Sigma_1$(mentioned in Remark $\ref{orientation}$). For constructing a triangulation of $\Sigma_g$, we implement an iterative connected sum process and proceed by performing a connected sum of a triangulated orientable surface of genus $(i+1)$ with a triangulated orientable surface of genus $1$.
 
 The triangulation of the genus $(i+1)$ surface is constructed in Theorem $\ref{thm1}$ for the simplicial degree $(2i+1)$ map from $\Sigma_{i+1}$ to $\Sigma_1$, with its vertex labeling consistent with Theorem $\ref{thm1}$. The genus $1$ surface is identical to that shown in Figure \ref{fig 1}, except that its vertex $v_j$ is replaced with $u_{(j,2i+2)}$ for $1\leq j\leq 7$. Now, we subdivide the triangulation of the genus $1$ surface by introducing an edge inside one of its triangles. Without loss of generality, assume that the chosen triangle is $[u_{(1,2i+2)}u_{(3,2i+2)}u_{(4,2i+2)}]$, and the vertices of the edge introduced inside this triangle are $u_{(4',2i+2)}$ and $u_{(3',2i+2)}$. After subdivision, the new triangles replacing $[u_{(1,2i+2)}u_{(3,2i+2)}u_{(4,2i+2)}]$ are $[u_{(1,2i+2)}u_{(3',2i+2)}u_{(4',2i+2)}]$, $[u_{(1,2i+2)}u_{(3,2i+2)}u_{(4',2i+2)}]$,  $[u_{(1,2i+2)}u_{(3',2i+2)}u_{(4,2i+2)}]$,  $[u_{(3,2i+2)}u_{(3',2i+2)}u_{(4',2i+2)}]$, and $[u_{(3,2i+2)}u_{(3',2i+2)}u_{(4,2i+2)}]$. The triangles involved in the connected sum between $\Sigma_{i+1}$ and $\Sigma_{1}$ are $[u_{(1,2i+2)}u_{(3',2i+2)}u_{(4',2i+2)}]$ from the triangulated $\Sigma_1$ and $[u_{(1,1)}u_{(3,2i+1)}u_{(4,1)}]$ from $\Sigma_{i+1}$. After this operation, we obtain a triangulated $\Sigma_{i+2}$. Again, we take the connected sum of the obtained triangulated surface with another triangulated orientable genus $1$ surface. The triangulation of an orientable genus $1$ surface will be the same as we defined earlier, but the vertices are now labeled as  $u_{(j,2i+3)}$ for $1\le j \le 7$. We again subdivide the triangulated genus 1 surface. Without loss of generality, assume the triangle is $[u_{(1,2i+3)}u_{(3,2i+3)}u_{(7,2i+3)}]$, and the edge introduced inside the triangle is  $[u_{(3',2i+3)}u_{(7',2i+3)}]$. After subdivision, the new triangles replacing $[u_{(1,2i+3)}u_{(3,2i+3)}u_{(7,2i+3)}]$ are $[u_{(1,2i+3)}u_{(3',2i+3)}u_{(7,2i+3)}]$, $[u_{(1,2i+3)}u_{(3',2i+3)}u_{(7',2i+3)}]$, $[u_{(1,2i+3)}u_{(3,2i+3)}u_{(7',2i+3)}]$, $[u_{(3,2i+3)}u_{(3',2i+3)}u_{(7',2i+3)}]$, and  $[u_{(3,2i+3)}u_{(3',2i+3)}u_{(7,2i+3)}]$. This subdivision gives the triangulation of an orientable surface of genus 1. The triangles that participate in the connected sum between $\Sigma_{i+2}$ and $\Sigma_{1}$ are $[u_{(1,2i+3)}u_{(3',2i+3)}u_{(7',2i+3)}]$ from the genus 1 surface and $[u_{(1,1)}u_{(3,2i+2)}u_{(7,2i+2)}]$ from the triangulated $\Sigma_{i+2}$. After this operation, we obtain $\Sigma_{i+3}$. 
This iterative process continues, wherein at each step, the selected triangle for the triangulated genus $1$ surface alternates. For instance, in the next step, it becomes $[u_{(1,2i+4)}u_{(3',2i+4)}u_{(4',2i+4)}]$. Repeating this process $g-i-1$ times, we obtain the desired triangulation of $\Sigma_g$.

Finally, we define a simplicial map $f: \Sigma_g \rightarrow \Sigma_1$ as follows. For $2 \le p\le 7$, $1\le q\le g+i$, map $f(u_{(p,q)}) = v_p$, and $u_{(1,1)}$ $\to$ $v_1$.

\begin{figure}[ht]
\tikzstyle{ver}=[]
\tikzstyle{vertex}=[circle, draw, fill=black!60, inner sep=0pt, minimum width=4pt]
\tikzstyle{edge} = [draw,black,thick,-]
\centering

\begin{tikzpicture}[scale=0.9]
\begin{scope}[shift={(0,0)}]
\foreach \x/\y/\z in {0/0/l_1, 1/0/l_2, 2/0/l_3, 3/0/s_2, 3/1/r_5, 3/2/r_4, 3/3/s_1, 2/3/u_3, 1/3/u_2, 0/3/u_1, 0/2/l_4, 0/1/l_5, 2/2/u_6, 2/1/u_7}{
\node[vertex] (\z) at (\x,\y){};}

\foreach \x/\y/\z in {-0.7/-0.1/u_{(1,1)},0.9/-0.35/u_{(2,1)},2.1
/-0.35/u_{(3,1)},3.7/-0.2/u_{(1,1)}, -0.7/3.1/u_{(1,1)},0.9/3.3/u_{(2,1)},2.1/3.3/u_{(3,1)},3.7/3/u_{(1,1)}, -0.7/1/u_{(5,1)},-0.7/2/u_{(4,1)},1.85/1.2/\tiny{u_{(7,1)}},1.5/2/u_{(6,1)},3.7/1/u_{(5,1)},3.7/2/u_{(4,1)}}{\node[ver] () at (\x,\y){{\fontsize{8pt}{7.6pt}\selectfont $\z$}};}

 \foreach \x/\y in {l_1/l_2,l_2/l_3, l_3/s_2, s_2/r_5, r_5/r_4, r_4/s_1, s_1/u_3, u_3/u_2, u_2/u_1, u_1/l_4, l_4/l_5,l_5/l_1,l_1/u_6,u_6/l_2,l_2/u_7, l_3/u_7,u_7/u_6,u_6/u_3,u_3/r_4,r_4/u_7, u_7/r_5,u_7/s_2,r_4/u_6,u_6/l_5,l_5/u_3,l_5/u_2,u_2/l_4}{
 \path[edge] (\x) -- (\y);}
 \fill[gray, opacity=0.4] (2,3) -- (3,2) -- (3,3) -- cycle;
 \node at (0.4,2.7) {\textcolor{black!90}{ \small \textbf{$+$}}};
  \node at (0.4,2.2) {\textcolor{black!90}{ \small \textbf{$-$}}};
  \node at (0.85,2.2) {\textcolor{black!90}{ \small \textbf{$+$}}};

  \node at (0.79,1.6) {\textcolor{black!90}{ \small \textbf{$-$}}};
  \node at (0.7,1) {\textcolor{black!90}{ \small \textbf{$+$}}};
  \node at (0.79,0.5) {\textcolor{black!90}{ \small \textbf{$-$}}};
  \node at (1.5,0.7) {\textcolor{black!90}{ \small \textbf{$+$}}};
  \node at (1.6,0.4) {\textcolor{black!90}{ \small \textbf{$-$}}};
  \node at (2.4,0.3) {\textcolor{black!90}{ \small {$+$}}};
    \node at (2.6,0.7) {\textcolor{black!90}{ \small {$-$}}};
  \node at (2.6,1.3) {\textcolor{black!90}{ \small {$+$}}};
   \node at (2.4,1.7) {\textcolor{black!90}{ \small {$-$}}};
   \node at (2.5,2.2) {\textcolor{black!90}{ \small {$+$}}};
   \node at (2.59,2.7) {\textcolor{black!90}{ \small {$-$}}};
\end{scope}
\begin{scope}[shift={(4.69,0)}]
  \draw[line width=0.9mm,black!60] (0,1)--(0,2); 
  \draw[line width=0.9mm,black!60] (0.5,1)--(0.5,2);
  \draw[line width=0.9mm,black!60] (-0.3,1.3)--(0.8,1.3); 
  \draw[line width=0.9mm,black!60] (-0.3,1.7)--(0.8,1.7);
 \end{scope}
\begin{scope}[shift={(7,0)}]
\foreach \x/\y/\z in {0/0/t_1, 1/0/l_2, 2/0/l_3, 4.1/0/s_2, 4.1/1/r_5, 4.1/1.9/r_4, 4.5/3.3/s_1, 2/3.3/u_3, 1/3.3/u_2, 0/3.3/u_1, 0/2/p_4, 0/1/l_5, 2/2/u_6, 2/1/u_7, 3.6/3/u_8, 3.7/2.5/q}{
\node[vertex] (\z) at (\x,\y){};}
\foreach \x/\y/\z in {-0.7/-0.1/u_{(1,2)},0.9/-0.35/u_{(2,2)},2.25/-0.35/u_{(3,2)},4.1/-0.35/u_{(1,2)}, -0.7/3.3/u_{(1,2)},0.9/3.6/u_{(2,2)},2.26/3.6/u_{(3,2)},4.7/3.6/u_{(1,2)}, -0.7/1/u_{(5,2)},-0.7/2/u_{(4,2)},1.85/1.2/\tiny{u_{(7,2)}},1.5/2/u_{(6,2)},4.78/1/u_{(5,2)},4.78/1.78/u_{(4,2)}, 4.4/2.4/u_{(3',2)}}{\node[ver] () at (\x,\y){{\fontsize{8pt}{7.6pt}\selectfont $\z$}};}
\foreach \x/\y/\z in {3/3/u_{(4',2)}}{\node[ver] () at (\x,\y){{\fontsize{4pt}{6pt}\selectfont $\z$}};}
\foreach \x/\y in {t_1/l_2, l_2/l_3, l_3/s_2, s_2/r_5, r_5/r_4, r_4/s_1, s_1/u_3, u_3/u_2, u_2/u_1, u_1/p_4, p_4/l_5, l_5/t_1, t_1/u_6, u_6/l_2, l_2/u_7, l_3/u_7, u_7/u_6, u_6/u_3, u_3/r_4, r_4/u_7, u_7/r_5, u_7/s_2, r_4/u_6, u_6/l_5, l_5/u_3, l_5/u_2, u_2/p_4, s_1/u_8, u_8/q, s_1/q,u_3/u_8,u_3/q,q/r_4}{
 \path[edge] (\x) -- (\y);}
 \fill[gray, opacity=0.4] (3.6,3) -- (3.7,2.5) -- (4.5,3.3) -- cycle;
  \node at (0.3,2.7) {\textcolor{black!90}{ \small \textbf{$+$}}};
  \node at (0.35,2.2) {\textcolor{black!90}{ \small \textbf{$-$}}};
  \node at (0.85,2.3) {\textcolor{black!90}{ \small \textbf{$+$}}};

  \node at (0.79,1.6) {\textcolor{black!90}{ \small \textbf{$-$}}};
  \node at (0.7,1) {\textcolor{black!90}{ \small \textbf{$+$}}};
  \node at (0.79,0.5) {\textcolor{black!90}{ \small \textbf{$-$}}};
  \node at (1.5,0.7) {\textcolor{black!90}{ \small \textbf{$+$}}};
  \node at (1.68,0.4) {\textcolor{black!90}{ \small \textbf{$-$}}};
  \node at (2.4,0.3) {\textcolor{black!90}{ \small {$+$}}};
    \node at (3,0.7) {\textcolor{black!90}{ \small {$-$}}};
  \node at (3.1,1.2) {\textcolor{black!90}{ \small {$+$}}};
   \node at (2.8,1.7) {\textcolor{black!90}{ \small {$-$}}};
   \node at (2.7,2.3) {\textcolor{black!90}{ \small {$+$}}};
   \node at (3.66,2.3) {\textcolor{black!90}{ \small {$+$}}};
    \node at (4.1,2.6) {\textcolor{black!90}{ \small {$-$}}};
    \node at (3.8,2.8) {\textcolor{black!90}{ \small {$+$}}};
     \node at (3.7,3.2) {\textcolor{black!90}{ \small {$-$}}};
      \node at (3.4,2.8) {\textcolor{black!90}{ \small {$-$}}};
   
\end{scope}

 \end{tikzpicture}
 \end{figure}

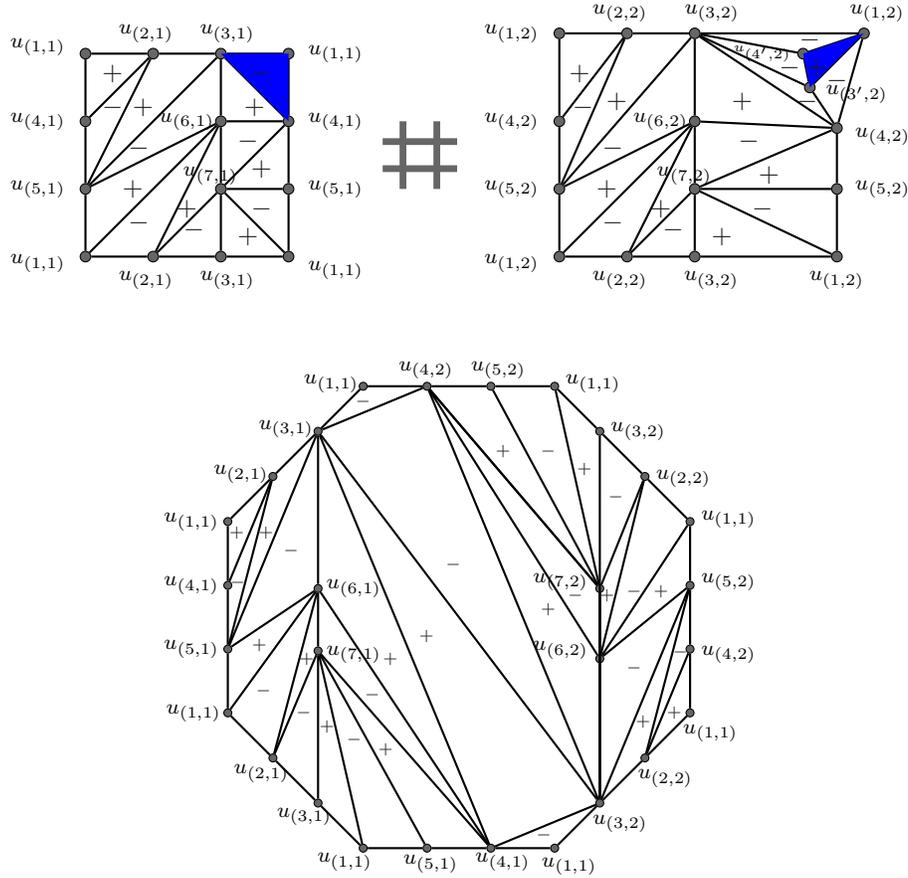
\begin{figure}[H]
\tikzstyle{ver}=[]
\tikzstyle{vertex}=[circle, draw, fill=black!60, inner sep=0pt, minimum width=3 pt]
\tikzstyle{edge} = [draw,black,thick,-]
\centering

\begin{tikzpicture}[scale=0.6]
\begin{scope}[shift={(0,0)}]
\foreach \x/\y/\z in {0/0/l_1,1/-1/l_2,2/-2/l_3,3/-3/l_4,4.414/-3/l_5,5.828/-3/l_6, 7.242/-3/l_7,8.242/-2/l_8,9.242/-1/l_9,10.242/0/l_10,10.242/1.414/l_11,10.242/2.828/l_12,10.242/4.242/l_13,9.242/5.242/l_14,8.242/6.242/l_15,7.242/7.242/l_16,5.828/7.242/l_17,4.414/7.242/l_18,3/7.242/l_19,2/6.242/l_20,1/5.242/l_21,0/4.242/l_22,0/2.828/l_23,0/1.4141/l_24,2/2.7572/l_25,2/1.3736/l_26,8.242/2.7572/l_27,8.242/1.2/l_28}{
\node[vertex] (\z) at (\x,\y){};}

\foreach \x/\y/\z in {-0.75/0/u_{(1,1)},0.75/-1.34/u_{(2,1)},1.7/-2.3/u_{(3,1)},2.6/-3.3/u_{(1,1)},4.5/-3.35/u_{(5,1)},6.1/-3.3/u_{(4,1)}, 7.57/-3.4/u_{(1,1)},8.7/-2.45/u_{(3,2)},9.7/-1.45/u_{(2,2)},10.7/-0.45/u_{(1,1)},11.1/1.3/u_{(4,2)},11.1/2.9/u_{(5,2)},11.1/4.3/u_{(1,1)},10.1/5.2/u_{(2,2)},9.1/6.242/u_{(3,2)},8.1/7.3/u_{(1,1)},6/7.6/u_{(5,2)},4.4/7.6/u_{(4,2)},2.3/7.3/u_{(1,1)},1.3/6.3/u_{(3,1)},0.3/5.3/u_{(2,1)},-0.8/4.242/u_{(1,1)},-0.8/2.828/u_{(4,1)},-0.8/1.4141/u_{(5,1)},2.8/2.7572/u_{(6,1)},2.8/1.3/u_{(7,1)},7.38/1.4/u_{(6,2)},7.38/2.9/u_{(7,2)}}{\node[ver] () at (\x,\y){{\fontsize{8pt}{7.6pt}\selectfont $\z$}};}

 \foreach \x/\y in {l_1/l_2,l_2/l_3,l_3/l_4,l_4/l_5,l_5/l_6,l_6/l_7,l_7/l_8,l_8/l_9,l_9/l_10,l_10/l_11,l_11/l_12,l_12/l_13,l_13/l_14,l_14/l_15,l_15/l_16,l_16/l_17,l_17/l_18,l_18/l_19,l_19/l_20,l_20/l_21,l_21/l_22,l_22/l_23,l_23/l_24,l_24/l_1,l_20/l_25,l_25/l_26,l_26/l_3,l_18/l_27,l_27/l_28,l_1/l_25,l_25/l_24,l_24/l_20,l_24/l_21,l_21/l_23,l_26/l_4,l_26/l_5,l_26/l_6,l_25/l_6,l_20/l_6,l_20/l_18,l_8/l_20,l_18/l_8,l_2/l_25,l_2/l_26,l_6/l_8,l_8/l_27,l_8/l_28,l_12/l_28,l_14/l_27, l_15/l_8, l_8/l_12,l_9/l_11, l_9/l_12,l_28/l_18,l_27/l_18,l_27/l_17,l_27/l_16,l_28/l_13,l_28/l_14}{
 \path[edge] (\x) -- (\y);}
 \node at (0.2,4) {\textcolor{black!90}{ \tiny\textbf{$+$}}};
  \node at (0.22,2.9) {\textcolor{black!90}{ \tiny \textbf{$-$}}};
  \node at (0.85,4) {\textcolor{black!90}{ \tiny \textbf{$+$}}};

  \node at (1.4,3.6) {\textcolor{black!90}{ \tiny \textbf{$-$}}};
  \node at (0.7,1.5) {\textcolor{black!90}{ \tiny\textbf{$+$}}};
  \node at (0.79,0.5) {\textcolor{black!90}{ \tiny \textbf{$-$}}};
  \node at (1.75,1.2) {\textcolor{black!90}{ \tiny\textbf{$+$}}};
  \node at (1.68,0) {\textcolor{black!90}{ \tiny \textbf{$-$}}};
  \node at (2.2,-0.3) {\textcolor{black!90}{ \tiny{$+$}}};
   \node at (2.8,-0.6) {\textcolor{black!90}{ \tiny{$-$}}};
    \node at (3.5,-0.8) {\textcolor{black!90}{ \tiny{$+$}}};
    \node at (3.2,0.4) {\textcolor{black!90}{ \tiny {$-$}}};
  \node at (3.6,1.2) {\textcolor{black!90}{ \tiny {$+$}}};
   \node at (4.4,1.7) {\textcolor{black!90}{ \tiny{$+$}}};
   \node at (5,3.3) {\textcolor{black!90}{ \tiny {$-$}}};
    \node at (3,6.9) {\textcolor{black!90}{ \tiny {$-$}}};
     \node at (7,-2.7) {\textcolor{black!90}{ \tiny {$-$}}};
   \node at (7.1,2.3) {\textcolor{black!90}{ \tiny{$+$}}};
    \node at (7.7,2.6) {\textcolor{black!90}{ \tiny {$-$}}};
    \node at (6.1,5.8) {\textcolor{black!90}{ \tiny {$+$}}};
    \node at (7.1,5.8) {\textcolor{black!90}{ \tiny{$-$}}};
    \node at (7.9,5.4) {\textcolor{black!90}{ \tiny{$+$}}};
    \node at (8.6,4.8) {\textcolor{black!90}{ \tiny{$-$}}};
    \node at (8.4,2.6) {\textcolor{black!90}{ \tiny{$+$}}};
    \node at (9,2.7) {\textcolor{black!90}{ \tiny{$-$}}};
     \node at (9.6,2.7) {\textcolor{black!90}{ \tiny{$+$}}};
     \node at (9,1) {\textcolor{black!90}{ \tiny{$-$}}};
     \node at (9.2,-0.2) {\textcolor{black!90}{ \tiny{$+$}}};
     \node at (10.03,1.34) {\textcolor{black!90}{ \tiny{$-$}}};
      \node at (9.89,0) {\textcolor{black!90}{ \tiny{$+$}}};
\end{scope}

\end{tikzpicture}
\caption{Triangulation of $\Sigma_2$ for a degree $2$ simplicial map}
\label{fig 4}
\end{figure}

This simplicial map induces a homomorphism $f_*: H_2(\Sigma_g, \mathbb{Z})\to H_2(\Sigma_1, \mathbb{Z})$. We claim that the deg($f$) $=$ $g+i$. To justify this, assume the simplex $[v_1v_2v_4]$ is a positive simplex in $\Sigma_1$ and $[u_{(1,1)}u_{(2,1)}u_{(4,1)}]$ is a positive simplex in $\Sigma_g$. In each connected sum operation, we observe that two triangles map to an edge, which contributes nothing to the degree. To avoid ambiguity arising from such cases, we focus on the interiors of the triangles. From the triangulation and simplicial map construction, the preimage of the interior of each 2-simplex contains the interiors of $g+i$ disjoint simplices, all preserving the same sign as their image triangle has. Hence, the degree of $f$ is $g+i$.

We refer to Figure $\ref{fig 4}$, which shows a triangulation of $\Sigma_2$, and which we use to define a simplicial map $f: \Sigma_2 \to \mathbb{T}^2$ of degree 2. The map sends $u_{(1,1)}$ to $v_1$ and $u_{(i,j)}$ to $v_i$ for $2 \leq i \leq 7$ and $j \in {1,2}$. By Corollary $\ref{Coro.}$, this is a minimal degree 2 simplicial map.

We fix orientations by taking $[v_1v_2v_4]$ and $[u_{(1,1)}u_{(2,1)}u_{(4,1)}]$ as positively oriented in $\mathbb{T}^2$ and $\Sigma_2$, respectively. The orientation of each simplex in $\Sigma_2$ is indicated in Figure $\ref{fig 4}$.

Two simplices in $\Sigma_2$, $[u_{(3,1)}u_{(3,2)}u_{(4,2)}]$ and $[u_{(3,1)}u_{(3,2)}u_{(4,1)}]$ (considered in dictionary order), map to the edge $[v_3v_4]$ and do not contribute to the map’s degree. Each triangle in $\mathbb{T}^2$ has exactly two disjoint preimages in $\Sigma_2$, matching in orientation, except for triangles like $[v_1v_3v_4]$ and $[v_3v_4v_6]$, where some preimages map to edges. To avoid ambiguity, we consider only the preimages of the interiors of triangles. From this, we conclude that the algebraic number of preimages is 2 for every triangle $\sigma \in \mathbb{T}^2$.
\end{proof}

\begin{theorem}\label{thm 11}
  For integers $|d|=g-i$, where $ 1 \leq i\leq g-1$, and $g\geq 2$, there exists a degree $d$ simplicial map from a $(6g-2i+1)$-vertex triangulated genus $g$ orientable surface to a $7$-vertex triangulated genus 1 orientable surface. 
 \end{theorem}
  \begin{proof}
  For a simplicial map $f:\Sigma_g \to \Sigma_1$ of degree $d\leq 0$, the idea of the construction has already been given at the beginning of the proof of Theorem $\ref{thm1}$. Therefore, our focus is on constructing a degree $d>0$ simplicial map from a triangulated $\Sigma_g$ to $\Sigma_1$. 

\begin{figure}[H]
\tikzstyle{ver}=[]
\tikzstyle{vertex}=[circle, draw, fill=black!60, inner sep=0pt, minimum width=3.5 pt]
\tikzstyle{edge} = [draw,black,thick,-]
\centering

\begin{tikzpicture}[scale=0.5]
\begin{scope}[shift={(0,0)}]
\foreach \x/\y/\z in {0/0/l_1,1/-1/l_2,2/-2/l_3,3/-3/l_4,4.414/-3/l_5,5.828/-3/l_6, 7.242/-3/l_7,8.242/-2/l_8,9.242/-1/l_9,10.242/0/l_10,10.242/1.414/l_11,10.242/2.828/l_12,10.242/4.242/l_13,9.242/5.242/l_14,8.242/6.242/l_15,7.242/7.242/l_16,5.828/7.242/l_17,4.414/7.242/l_18,3/7.242/l_19,2/6.242/l_20,1/5.242/l_21,0/4.242/l_22,0/2.828/l_23,0/1.4141/l_24,2/2.7572/l_25,2/1.3736/l_26,8.242/2.7572/l_27,8.242/1.2/l_28}{
\node[vertex] (\z) at (\x,\y){};}

\foreach \x/\y/\z in {-0.75/0/u_{(1,1)},0.75/-1.34/u_{(2,1)},1.7/-2.3/u_{(3,1)},2.6/-3.3/u_{(1,1)},4.5/-3.35/u_{(5,1)},6.1/-3.3/u_{(4,1)}, 7.57/-3.4/u_{(1,1)},8.7/-2.45/u_{(3,2)},9.7/-1.45/u_{(2,2)},10.7/-0.45/u_{(1,1)},11.1/1.3/u_{(4,2)},11.1/2.9/u_{(5,2)},11.1/4.3/u_{(1,1)},10.1/5.2/u_{(2,2)},9.1/6.242/u_{(3,2)},8.1/7.3/u_{(1,1)},6/7.6/u_{(5,2)},4.4/7.6/u_{(4,2)},2.3/7.3/u_{(1,1)},1.3/6.3/u_{(3,1)},0.3/5.3/u_{(2,1)},-0.8/4.242/u_{(1,1)},-0.8/2.828/u_{(4,1)},-0.8/1.4141/u_{(5,1)},2.8/2.7572/u_{(6,1)},2.8/1.3/u_{(7,1)},7.38/1.4/u_{(6,2)},7.38/2.9/u_{(7,2)}}{\node[ver] () at (\x,\y){{\fontsize{8pt}{7.6pt}\selectfont $\z$}};}

 \foreach \x/\y in {l_1/l_2,l_2/l_3,l_3/l_4,l_4/l_5,l_5/l_6,l_6/l_7,l_7/l_8,l_8/l_9,l_9/l_10,l_10/l_11,l_11/l_12,l_12/l_13,l_13/l_14,l_14/l_15,l_15/l_16,l_16/l_17,l_17/l_18,l_18/l_19,l_19/l_20,l_20/l_21,l_21/l_22,l_22/l_23,l_23/l_24,l_24/l_1,l_20/l_25,l_25/l_26,l_26/l_3,l_18/l_27,l_27/l_28,l_1/l_25,l_25/l_24,l_24/l_20,l_24/l_21,l_21/l_23,l_26/l_4,l_26/l_5,l_26/l_6,l_25/l_6,l_20/l_6,l_20/l_18,l_8/l_20,l_18/l_8,l_2/l_25,l_2/l_26,l_6/l_8,l_8/l_27,l_8/l_28,l_12/l_28,l_14/l_27, l_15/l_8, l_8/l_12,l_9/l_11, l_9/l_12,l_28/l_18,l_27/l_18,l_27/l_17,l_27/l_16,l_28/l_13,l_28/l_14}{
 \path[edge] (\x) -- (\y);}
 \fill[gray, opacity=0.4] (1,5.242)--(0,4.242)--(0,2.828)-- cycle;
\end{scope}
\begin{scope}[shift={(13,1)}]
  \draw[line width=0.5mm,black!70] (0,1)--(0,2); 
  \draw[line width=0.5mm,black!70] (0.5,1)--(0.5,2);
  \draw[line width=0.5mm,black!70] (-0.3,1.3)--(0.8,1.3); 
  \draw[line width=0.5mm,black!70] (-0.3,1.7)--(0.8,1.7);
 \end{scope}
\begin{scope}[shift={(16,-1)}]
\foreach \x/\y/\z in {0/0/l_1, 2.5/0/l_2, 5/0/l_3, 7.5/0/s_2,7.5/2.5/r_5, 7.5/5/r_4, 7.5/7.5/s_1, 5/7.5/u_3, 2.5/7.5/u_2, 0/7.5/u_1, 0/5/l_4, 0/2.5/l_5, 5/5/u_6, 5/2.5/u_7}{
\node[vertex] (\z) at (\x,\y){};}

\foreach \x/\y/\z in {-0.7/-0.1/u_{(1,3)},2.9/-0.5/u_{(2,3)},5.1
 /-0.5/u_{(3,3)},8.1/-0.5/u_{(1,3)}, -0.8/7.4/u_{(1,3)},2.9/7.8/u_{(2,3)},5.3/7.8/u_{(3,3)},7.7/7.8/u_{(1,3)}, -0.7/2.8/u_{(5,3)},-0.7/5.2/u_{(4,3)},4.2/2.6/\tiny{u_{(7,3)}},4.2/5.1/u_{(6,3)},8.3/2.6/u_{(5,3)},8.2/5.1/u_{(4,3)}}{\node[ver] () at (\x,\y){{\fontsize{9.5pt}{7.6pt}\selectfont $\z$}};}

 \foreach \x/\y in {l_1/l_2,l_2/l_3, l_3/s_2, s_2/r_5, r_5/r_4, r_4/s_1, s_1/u_3, u_3/u_2, u_2/u_1, u_1/l_4, l_4/l_5,l_5/l_1,l_1/u_6,u_6/l_2,l_2/u_7, l_3/u_7,u_7/u_6,u_6/u_3,u_3/r_4,r_4/u_7, u_7/r_5,u_7/s_2,r_4/u_6,u_6/l_5,l_5/u_3,l_5/u_2,u_2/l_4}{
 \path[edge] (\x) -- (\y);}
 \fill[gray, opacity=0.4] (0,7.5) -- (0,5) -- (2.5,7.5) -- cycle;
\end{scope}
\end{tikzpicture}
\end{figure}

     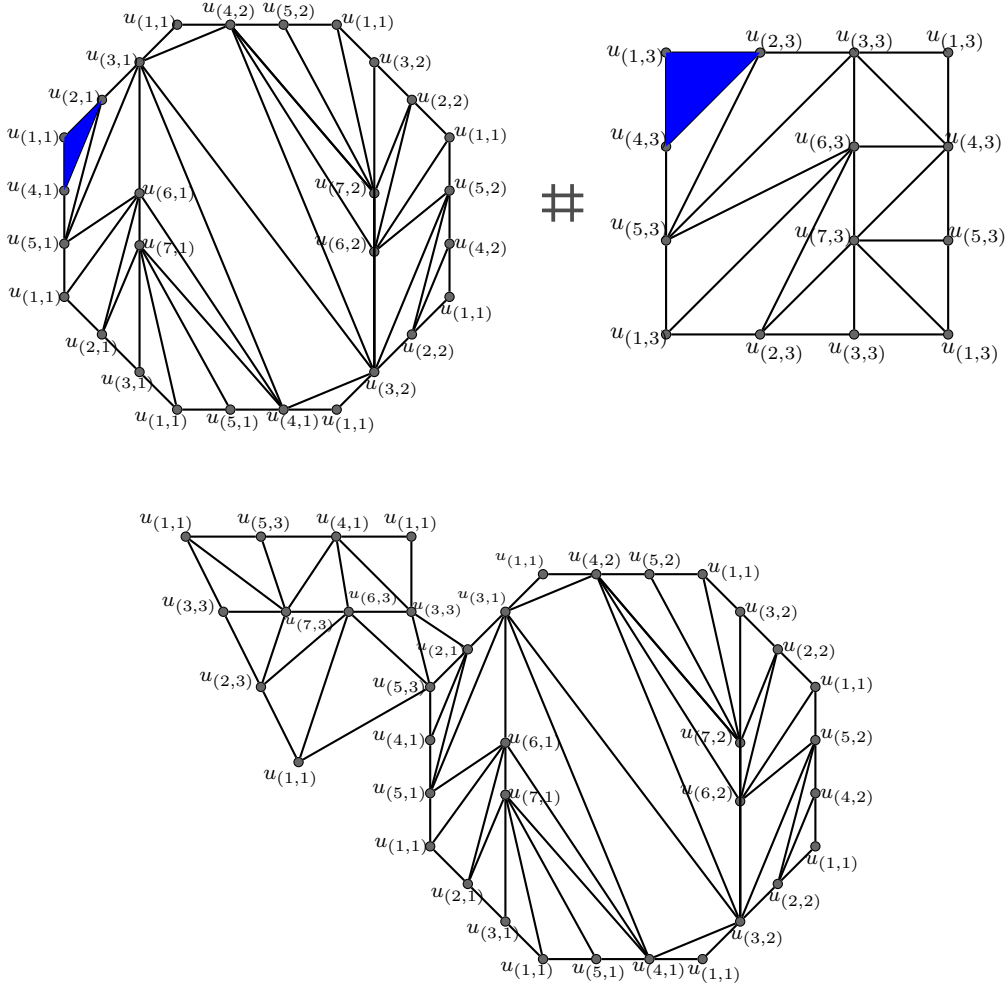
\begin{figure}[H]
\tikzstyle{ver}=[]
\tikzstyle{vertex}=[circle, draw, fill=black!60, inner sep=0pt, minimum width=3.5 pt]
\tikzstyle{edge} = [draw,black,thick,-]
\centering

\begin{tikzpicture}[scale=0.5]
\begin{scope}[shift={(0,0)}]
\foreach \x/\y/\z in {0/0/l_1,1/-1/l_2,2/-2/l_3,3/-3/l_4,4.414/-3/l_5,5.828/-3/l_6, 7.242/-3/l_7,8.242/-2/l_8,9.242/-1/l_9,10.242/0/l_10,10.242/1.414/l_11,10.242/2.828/l_12,10.242/4.242/l_13,9.242/5.242/l_14,8.242/6.242/l_15,7.242/7.242/l_16,5.828/7.242/l_17,4.414/7.242/l_18,3/7.242/l_19,2/6.242/l_20,1/5.242/l_21,0/4.242/l_22,0/2.828/l_23,0/1.4141/l_24,2/2.7572/l_25,2/1.3736/l_26,8.242/2.7572/l_27,8.242/1.2/l_28, -0.5/6.242/l_29,-0.5/8.242/l_30,-2.5/8.242/l_31,-4.5/8.242/l_32,-6.5/8.242/l_33,-5.5/6.242/l_34,-4.5/4.242/l_35,-3.5/2.242/l_36,-3.833/6.242/l_37, -2.167/6.242/l_38}{
\node[vertex] (\z) at (\x,\y){};}

\foreach \x/\y/\z in {-0.75/0/u_{(1,1)},0.75/-1.34/u_{(2,1)},1.7/-2.3/u_{(3,1)},2.6/-3.3/u_{(1,1)},4.5/-3.35/u_{(5,1)},6.1/-3.3/u_{(4,1)}, 7.57/-3.4/u_{(1,1)},8.7/-2.45/u_{(3,2)},9.7/-1.45/u_{(2,2)},10.7/-0.45/u_{(1,1)},11.1/1.3/u_{(4,2)},11.1/2.9/u_{(5,2)},11.1/4.3/u_{(1,1)},10.1/5.2/u_{(2,2)},9.1/6.242/u_{(3,2)},8.1/7.3/u_{(1,1)},6/7.6/u_{(5,2)},4.4/7.6/u_{(4,2)},-0.8/4.242/u_{(5,3)},-0.8/2.828/u_{(4,1)},-0.8/1.4141/u_{(5,1)},2.8/2.7572/u_{(6,1)},2.8/1.3/u_{(7,1)},7.38/1.4/u_{(6,2)},7.38/2.9/u_{(7,2)},-0.5/8.6/u_{(1,1)}, -2.3/8.6/u_{(4,1)},-4.4/8.6/u_{(5,3)},-7/8.6/u_{(1,1)},-6.4/6.4/u_{(3,3)},-5.4/4.4/u_{(2,1)},-3.7/1.8/u_{(1,1)}}{\node[ver] () at (\x,\y){{\fontsize{8pt}{7.6pt}\selectfont $\z$}};}

\foreach \x/\y/\z in {0.3/5.2/u_{(2,1)}, 2.4/7.6/u_{(1,1)}, 1.4/6.6/u_{(3,1)},0.2/6.3/u_{(3,3)},-1.5/6.6/u_{(6,3)}, -3.2/5.9/u_{(7,3)}}
{\node[ver] () at (\x,\y){{\fontsize{6.5pt}{5pt}\selectfont $\z$}};}

 \foreach \x/\y in {l_1/l_2,l_2/l_3,l_3/l_4,l_4/l_5,l_5/l_6,l_6/l_7,l_7/l_8,l_8/l_9,l_9/l_10,l_10/l_11,l_11/l_12,l_12/l_13,l_13/l_14,l_14/l_15,l_15/l_16,l_16/l_17,l_17/l_18,l_18/l_19,l_19/l_20,l_20/l_21,l_21/l_22,l_22/l_23,l_23/l_24,l_24/l_1,l_20/l_25,l_25/l_26,l_26/l_3,l_18/l_27,l_27/l_28,l_1/l_25,l_25/l_24,l_24/l_20,l_24/l_21,l_21/l_23,l_26/l_4,l_26/l_5,l_26/l_6,l_25/l_6,l_20/l_6,l_20/l_18,l_8/l_20,l_18/l_8,l_2/l_25,l_2/l_26,l_6/l_8,l_8/l_27,l_8/l_28,l_12/l_28,l_14/l_27, l_15/l_8, l_8/l_12,l_9/l_11, l_9/l_12,l_28/l_18,l_27/l_18,l_27/l_17,l_27/l_16,l_28/l_13,l_28/l_14, l_21/l_29, l_29/l_30, l_30/l_31, l_31/l_32, l_32/l_33, l_33/l_34, l_34/l_35, l_35/l_36, l_36/l_22, l_34/l_37, l_37/l_38,l_38/l_29,l_22/l_29,l_22/l_38,l_29/l_31,l_31/l_38,l_31/l_37,l_32/l_37,l_33/l_37,l_35/l_37,l_35/l_38,l_36/l_38}{
 \path[edge] (\x) -- (\y);}

\end{scope}
\end{tikzpicture}
\caption{Triangulation of $\Sigma_3$ for a degree $2$ simplicial map}
\label{fig 5}
\end{figure}
  
We begin by constructing the desired triangulation of $\Sigma_g$. This process starts with a triangulation of $\Sigma_{g-i}$, which we define in Theorem $\ref{thm 10}$ for a simplicial map of degree $g-i$ from $\Sigma_{g-i}$ to $\Sigma_1$. Now, we perform a connected sum between $\Sigma_{g-i}$ and a triangulated genus $1$ orientable surface containing $7$ vertices, shown in Figure $\ref{fig 1}$. Replace the vertices $v_j$ of $\Sigma_1$ with $u_{(j,g-i+1)}$, for $1\le j\le 7$, and select two triangles, one from each surface, to define the connected sum between $\Sigma_{g-i}$ and $\Sigma_1$. Without loss of generality, consider the triangles [$u_{(1,1)}u_{(2,1)}u_{(4,1)}]$ from $\Sigma_{g-i}$ and [$u_{(1,g-i+1)}u_{(2,g-i+1)}u_{(4,g-i+1)}]$ from the genus 1 surface. After this operation, we obtain a triangulated $\Sigma_{g-i+1}$. We repeat this process by performing another connected sum between a triangulated $\Sigma_{g-i+1}$ and a triangulated genus $1$ orientable surface with $7$ vertices, shown in Figure $\ref{fig 1}$. This time, we replace the vertices $v_j$ with $u_{(j,g-i+2)}$, for $1\le j\le 7$. Assume that the participating triangles are [$u_{(1,g-i+1)}u_{(5,g-i+1)}u_{(7,g-i+1)}]$ from $\Sigma_{g-i+1}$ and [$u_{(1,g-i+2)}u_{(5,g-i+2)}u_{(7,g-i+2)}]$ from the new genus 1 surface for performing a connected sum between $\Sigma_{g-i+1}$ and a triangulated genus $1$ orientable surface. After this procedure, we obtain a triangulated $\Sigma_{g-i+2}$. Continuing this procedure $i$ times, where we alternate the choice of the chosen triangle at each step, we ultimately construct a triangulation of the orientable surface of genus $g$ with $6g-2i+1$ vertices. Now we define the simplicial map $f: \Sigma_g \to \Sigma_1$ by sending $u_{(j,k)}$, where $1\le j\le 7$ and $1 \le k \le g-i$ to $v_j$, while sending all remaining vertices to $v_1$. 
  
We claim that the map $f$ has degree $d$. The map $f$ induces a homomorphism $f_*:H_2(\Sigma_g,\mathbb{Z})\rightarrow H_2 (\Sigma_1,\mathbb{Z})$. Let $[v_1v_2v_4]$ be a positive $2$-simplex in $\Sigma_1$, and let $[u_{(2,1)}u_{(3,1)}u_{(5,1)}]$ be a positive $2$-simplex in $\Sigma_g$. Using this choice, we define the orientation of every triangle in $\Sigma_g$, where the vertex order follows the dictionary order. 
From the simplicial map $f$, we observe that, except for $14(g-i)$ triangles, all others map either to a vertex or to an edge. Since such images do not contribute to the degree, we focus on the interiors of triangles in $\Sigma_1$ for the degree computation. It follows that the preimage of the interior of the triangle in $\Sigma_1$ consists of the interiors of $g-i$ disjoint triangles in $\Sigma_g$, each of which preserves the orientation(sign). Hence, the degree of $f$ is $g-i$.

We refer to Figure $\ref{fig 5}$, which shows a triangulation of $\Sigma_3$. This triangulation allows us to define a degree $2$ simplicial map $f:\Sigma_3 \to \Sigma_1$. The map $f$ is defined as in the proof of the theorem specifically, $u_{(i,j)}$ maps to $v_i$ for $1\leq i\leq 7$ and $j\in \{1,2\}$, while the remaining vertices $u_{(3,3)}$, $u_{(5,3)}$, $u_{(6,3)}$, and $u_{(7,3)}$ map to $v_1$. From the triangulation of $\Sigma_3$ and the simplicial map $f$,  we observe that the triangle $[u_{(5,3)}u_{(2,1)}u_{(4,1)}]$ maps to $[v_1v_2v_4]$, while the triangle $[u_{(1,1)}u_{(4,1)}u_{(3,3)}]$ maps to an edge $[v_1v_4]$. Any triangle that contains two vertices of the form $u_{(i,3)}$ maps either to an edge or a vertex. All other triangles map to triangles in $\Sigma_1$. Therefore, we see that the preimage of the interior of each triangle in triangulated $\Sigma_1$ contains the interiors of exactly three disjoint triangles in triangulated $\Sigma_3$, each preserving the sign of the triangles. Hence, the degree of the map $f$ is $2$. 
\end{proof}

 \begin{remark}
 
{\rm
With the help of the constructions of simplicial maps of degree $d$ from $\Sigma_g$ $\to$ $\Sigma_1$, where $g\geq 1$, we generalize the approach to define simplicial degree maps $f: \Sigma_{g_1} \rightarrow \Sigma_{g_2}$, where $g_2>1$ and $|d| \leq \lfloor \frac{g_1-1}{g_2-1}\rfloor$. The triangulations of $\Sigma_{g_1}$, and $\Sigma_{g_2}$ are as follows: If $g_1=g_2$, the possible degrees are 0 and $\pm 1$. For degree $1$, we can consider the simplicial map as the identity map between minimal triangulations of $\Sigma_g$, and for degree $-1$, we reverse the sign of simplices that lie in a domain. For degree $0$, any non-surjective simplicial map will work. When $g_1\neq g_2$, we consider two simplicial maps $g:\Sigma_{g_1-d} \rightarrow \Sigma_{g_2-1}$ and $h: \Sigma_1 \rightarrow \Sigma_1$ both of degree $d$. To obtain a triangulation of $\Sigma_{g_2}$, we perform a connected sum between a triangulated $\Sigma_{g_2-1}$(which lies in the codomain of $g$) and $\Sigma_1$(which lies in the codomain of $h$). Specifically, we remove the interior of a chosen triangle from $\Sigma_{g_2-1}$ and another from $\Sigma_1$, identifying their boundaries. Each triangle in $\Sigma_{g_2-1}$ and $\Sigma_1$ has $d$ triangles in its preimage under the maps $g$ and $h$, respectively. For constructing the triangulation of $\Sigma_{g_1}$, we perform the connected sum operation $d$ times between $\Sigma_{g_1-d}$(which lies in the domain of $g$) and $\Sigma_1$(which lies in the domain of $h$), ensuring that the operation is carried out on all preimage triangles corresponding to the ones selected for the connected sum of $\Sigma_{g_2-1}$ and $\Sigma_1$. Now, we have triangulated $\Sigma_{g_2}$ and $\Sigma_{g_1}$, and we define the simplicial map $f:\Sigma_{g_1}\rightarrow \Sigma_{g_2}$ by setting $f(v)=g(v)$ for vertices $v\in \Sigma_{g_1-d}$ and $f(v)=h(v)$ for vertices $v\in \Sigma_{1}$. By construction, the map $f$ is a simplicial map, and the degree of this simplicial map is $d$.}
 \end{remark} 
\section{Future Directions}
The construction of manifold triangulations and the search for minimal ones have been fundamental topics in combinatorial topology. In this paper, we focus on the unique minimal 7-vertex triangulation of the torus and construct simplicial degree \( d \) maps from a triangulated genus \( g \) surface to this 7-vertex torus triangulation for \( g \geq 1 \). Our construction is minimal for all \( d \) when \( g = 1,2 \), and for \( g \geq 3 \), it remains minimal for \( |d| \geq 2g - 1 \). The results obtained in this work, along with the methods and observations introduced, open several potential research directions.  

\begin{speculation} 
Given the $7$-vertex torus triangulation, one can explore the minimal simplicial maps of degree \( d \) from a genus \( g \) surface to the 7-vertex torus for \( |d| < 2g - 1 \) when \( g \geq 3 \).  
\end{speculation}  

\begin{speculation} 
The minimal triangulation of an orientable genus $2$ surface (double torus) requires at least $10$ vertices, with $865$ distinct triangulations known. One can choose any of these triangulations and investigate all possible minimal simplicial maps of degree \( d \) from a higher-genus surface of genus \( g \) to the double torus, where \( |d| \leq g - 1 \).
\end{speculation}  

The results can also be extended to higher dimensions. It is well known that the simplicial volume of \( \mathbb{S}^{n-1} \times \mathbb{S}^1 \) is zero, and a degree \( d \) self-map of \( \mathbb{S}^{n-1} \times \mathbb{S}^1 \) exists for every \( d \in \mathbb{Z} \). 

\begin{speculation} 
One can consider a minimal triangulation of \( \mathbb{S}^{n-1} \times \mathbb{S}^1 \) and investigate the existence of minimal simplicial self-maps of degree \( d \) on \( \mathbb{S}^{n-1} \times \mathbb{S}^1 \).  
\end{speculation}  

It is well known that the simplicial volume of the $n$-dimensional torus \( \mathbb{T}_n:=\mathbb{S}^{1} \times \mathbb{S}^1 \times \cdots \times  \mathbb{S}^1 \) ($n$-times) is zero, and a degree \( d \) self-map of \( \mathbb{T}_n \) exists for every \( d \in \mathbb{Z} \). 

\begin{speculation} 
One can consider a minimal triangulation of \( \mathbb{T}_n \) and investigate the existence of minimal simplicial self-maps of degree \( d \) on \( \mathbb{T}_n \).  
\end{speculation}  

These research problems highlight the many open problems and active research directions in this area, offering exciting opportunities for further exploration.

\medskip

 \noindent {\bf Acknowledgement:} The authors would like to thank the anonymous referees for their valuable comments and suggestions, which have greatly improved the manuscript. The first author is supported by the Mathematical Research Impact Centric Support (MATRICS) Research Grant (MTR/2022/000036) by ANRF (India). The second author is supported by the institute fellowship at IIT Delhi, India.
 
\smallskip

\noindent {\bf Data availability:} The authors declare that all data supporting the findings of this study are available within the article.

\smallskip

\noindent {\bf Declarations}

\noindent {\bf Conflict of interest:} No potential conflict of interest was reported by the authors.


\begin{thebibliography}{survey}
\bibitem{Amann} M. Amann, Degrees of Self-Maps of Simply Connected Manifolds, {\em Int. Math. Res. Not. IMRN}, {\bf 18} (2015), 8545--8589.

\bibitem{BGT25} Biplab Basak, Raju Kumar Gupta, and Ayushi Trivedi, Simplicial degree $d$ self-maps on $n$-spheres, 2025, 14 pages. https://arxiv.org/pdf/2409.00907

\bibitem{BrownSchirmer} R. F. Brown, H. Schirmer, Nielsen root theory and Hopf degree theory, {\em Pacific J. Math.}, {\bf 198 (1)} (2001), 49--80.


\bibitem{Edmonds} A. Edmonds, Deformation of maps to branched covering in dimension two, {\em Ann. of Math. (2)}, {\bf 110 (1)} (1979), 113--125.

\bibitem{Epstein1966} D. B. A. Epstein, The degree of a map, {\em Proc. London Math. Soc. (3)}, {\bf 16} (1966), 369--383.
 
\bibitem{Fan1967} K. Fan, Simplicial maps from an orientable $n$-pseudomanifold into $\Sp^{m}$ with the octahedral triangulation, {\em J. Combinatorial Theory}, {\bf 2} (1967), 588--602.


\bibitem{Gromov1982} M. Gromov, Volume and bounded cohomology, {\em Publications Math\'{e}matiques de l'IH\'{E}S}, {\bf 56} (1982), 5--99.



\bibitem{Hopf1928}  H. Hopf, Zur Topologie der Abbildungen von Mannigfaltigkeiten  I, {\em Math. Ann.}, {\bf 100 (1)} (1928), 579--608.

\bibitem{Hopf1930} H. Hopf, Zur Topologie der Abbildungen von Mannigfaltigkeiten II, {\em Math. Ann.}, {\bf 102 (1)}
 (1930), 562--623.

\bibitem{Kneser} H. Kneser, Die kleinste Bedeckungszahl innerhalb einer Klasse von Flächenabbildungen, {\em Math. Ann.}, {\bf 103 (1)} (1930), 347--358.

\bibitem{Madahar2001} K. V. Madahar and K. S. Sarkaria, Minimal simplicial self-maps of the 2-sphere, {\em Geom. Dedicata}, {\bf 84 (1-3)} (2001), 25--33.

 \bibitem{Milnor} J. W. Milnor, Topology from the differentiable viewpoint,{ University of Virginia Press, Charlottesv ille} (1965).


\bibitem{Olum1953} P. Olum, Mappings of manifolds and the notion of degree, {\em Ann. of Math. (2)}, {\bf 58} (1953), 458--480.



\bibitem{Ryabichev2024} A. Ryabichev, Maximal degree of a map of surfaces, {\em Pacific J. Math}, {\bf 328 (1)} (2024), 145--156.


\bibitem{Skora} R. Skora, The degree of a map between surfaces, {\em Math. Ann.}, {\bf 276 (3)} (1987), 415--423.

\bibitem{Whitney} H. Whitney, On the maps of an $n$-sphere into another $n$-sphere, {\em Duke Math. J.}, {\bf 3 (1)} (1937), 46--50.

\end{thebibliography}
\end{document}